\documentclass[12pt]{article}
\usepackage{latexsym,amsmath,amsthm, amssymb}
\usepackage{hyperref}
\usepackage{a4wide}
\usepackage{marginnote}
\usepackage{color}

\newcommand{\ind}{\mathrm{1}\hskip -3.2pt \mathrm{I}} 
\newcommand{\id}{\mbox{\rm Id}}

\newtheorem{theorem}{Theorem}
\newtheorem{proposition}[theorem]{Proposition}

\newtheorem{corollary}[theorem]{Corollary}

\theoremstyle{remark} 

\def\tr{\mathop{\rm Tr}\nolimits}
\def\va{\mathop{\rm Var}\nolimits}

\def\R{\mathbb R}

\def\vol{\mathop{\rm vol}\nolimits}
\def\Ric{\mathop{\rm Ric}\nolimits}

\numberwithin{equation}{section}

\title{Dimensional variance inequalities of Brascamp-Lieb type and a local approach to dimensional Pr\'ekopa's theorem}

\author{ Van Hoang Nguyen\footnote{Institut de Math\'ematiques de Jussieu, Universit\'e Pierre et Marie Curie (Paris 6), 4 place Jussieu, 75252 Paris, France. Email: vanhoang@math.jussieu.fr} }

\begin{document}
\maketitle


\renewcommand{\thefootnote}{}

\footnote{Supported by the French ANR GeMeCoD, ANR 2011 BS01 007 01.}

\footnote{2010 \emph{Mathematics Subject Classification\text}: 26A51, 52A40.}

\footnote{\emph{Key words and phrases\text}: Dimensional variance inequalities, Brascamp-Lieb type inequalities, sharp weighted Poincar\'e inequalities, convex measures, $L^2$-method}

\renewcommand{\thefootnote}{\arabic{footnote}}
\setcounter{footnote}{0}

\begin{abstract}
We give a new approach, inspired by H\"ormander's $L^2$-method,  to weighted variance inequalities which extend results obtained by Bobkov and Ledoux. It provides in particular a local proof of the dimensional functional forms of the Brunn-Minkowski inequalities. We also present several applications of these variance inequalities, including reverse H\"older inequalities for convex functions, weighted Brascamp-Lieb inequalities and sharp weighted Poincar\'{e} inequalities for generalized Cauchy measures.  
\end{abstract}

\section{Introduction}

Our original motivation was to provide a local $L^2-$proof of the dimensional Pr\'{e}kopa inequality (Theorem~\ref{extensionofPrekopa} below). This inequality comes from a functional form of the Brunn-Minkowski inequality
\begin{equation}\label{eq:Brunn-Minkowski}
|A+B|^{\frac{1}{n}}\geq |A|^{\frac{1}{n}} + |B|^{\frac{1}{n}},
\end{equation}
where $A$ and $B$ are two Borel (later convex) subsets of $\R^n$, and
$ A+B=\{a+b: a\in A \text{ and } b\in B\}$.
Eventually, we end up establishing the  following two Theorems, which  are the main new results of the present paper. In the sequel, we fix a Euclidean structure $\langle \cdot , \cdot\rangle$ on $\R^n$ with respect to which Hessians $D^2 $ and gradients $\nabla$ are computed; 
recall the notation $\va_{\mu}(f):=\int f^2\, d\mu- (\int f\, d\mu)^2$ for the variance of a function $f$ with respect to a probability measure $\mu$.
 
\begin{theorem}\label{maintheorem1}
Given $\beta, r\in \R$ such that $ \beta > r + (n + \sqrt{n^2 + 4(r^2 -r)n})/2\in [n,+\infty)$, set 
$$ A(n,\beta,r) := \frac{\beta - n}{n} - \frac{(n - 1) r^2}{n(\beta - 2r)} > 0. $$
 Let $\varphi$ be a positive $C^2$ convex function defined on an open convex set $\Omega\subseteq\R^n$,  such that $d\mu_\beta = \varphi(x)^{-\beta}dx$ is a probability measure on $\Omega$. Then, for any locally Lipschitz $f\in L^2(\mu_\beta)$, setting $g = f\varphi^{1-r}$, we have
\begin{equation}\label{eq:convexmainresult}
(\beta - 2r +1)\va_{\mu_\beta}(f)\leq \int_\Omega \frac{\langle (D^2\varphi)^{-1}\nabla g,\nabla g\rangle}{\varphi}\, \varphi^{2r}\, d\mu_\beta + \frac{(1 - r)^2}{A(n,\beta,r)} \bigg(\int_\Omega f\, d\mu_\beta\bigg)^2. 
\end{equation}
\end{theorem}

In general, this theorem is applied with $\Omega = \R^n$.  We also have a parallel "concave" case of the Theorem \ref{maintheorem1} which is:

\begin{theorem}\label{maintheorem2}
Given $\beta, r\in \R$ satisfying $ \beta > -r + (-n + \sqrt{n^2 + 4(r^2 -r)n})/2\in [-1,+\infty)$, set
$$ B(n,\beta,r) := \frac{\beta + n}{n} - \frac{(n - 1) r^2}{n(\beta + 2r)} > 0. $$
 Let $\varphi$ be a positive concave  $C^2$ function defined on a \emph{bounded} open convex set $\Omega\subset \R^n$ such that $d\nu_\beta = \varphi(x)^{\beta}\ind_{\Omega}(x)dx$ is a probability measure. Then, for any locally Lipschitz $f\in L^2(\nu_\beta)$, setting $g = f\varphi^{1-r}$, we  have
\begin{equation}\label{eq:concavemainresult}
(\beta + 2r -1)\va_{\nu_\beta}(f)\leq \int_\Omega \frac{\langle (-D^2\varphi)^{-1}\nabla g,\nabla g\rangle}{\varphi}\, \varphi^{2r}\, d\nu_\beta + \frac{(1 - r)^2}{B(n,\beta,r)} \bigg(\int_\Omega f\, d\nu_\beta\bigg)^2. 
\end{equation}
\end{theorem}

One can easily check that equality holds in~\eqref{eq:convexmainresult} and~\eqref{eq:concavemainresult} when $f= \langle \nabla \varphi, z_0\rangle\varphi^{r-1}$, for some fixed $z_0\in \R^n$.

In order to explain and motivate these results, but also to justify and understand the geometric nature of the conditions on the parameters, we need to step back a moment to the Brunn-Minkowski inequalities.

 Inequality \eqref{eq:Brunn-Minkowski} says that $|\cdot|^{1/n}$ is concave; in the terminology recalled below, it means that the Lebesgue measure is $1/n$-concave.
Using the homogeneity of Lebesgue measure, one can easily check that the inequality~\eqref{eq:Brunn-Minkowski} is equivalent to the following \emph{a-dimensional} inequality: for every $A,B\in \R^n$ and $t\in [0,1]$,
\begin{equation}\label{eq:BMineq}
| tA+(1-t)B|\geq |A|^t|B|^{1-t},\quad \forall t\in \left [0,1\right] .
\end{equation}
Inequality~\eqref{eq:BMineq} says that Lebesgue measure $|\cdot|$ on $\R^n$ is log-concave. More generally, a Borel measure $\mu$ on $\R^n $ is said to be log-concave if it satisfies
$$ \mu(tA+(1-t)B)\geq \mu(A)^t\mu(B)^{1-t}, $$
for any $0\leq t\leq 1$ and $A, B$  two Borel sets of $\R^n$. 
The \emph{dimensional} analogues of this property are defined as follows.

We introduce first, for $\kappa \in \R \cup\{\pm \infty\}$, $t\in [0,1]$ and $a,b \ge 0$, the $\kappa$-mean
$$
\mathcal{M}_{t}^{\kappa}(a,b)=
(ta^{\kappa}+(1-t)b^{\kappa})^{\frac{1}{\kappa}}
$$
with the convention that $\mathcal{M}_t^{\kappa}(a,b)=0$ if $a b = 0$.
The extremal cases are defined in the limit by $\mathcal M_t^{0}(a,b)= a^t b^{1-t}$, $\mathcal M_t^{-\infty}(a,b)= \min\{a,b\}$, and  $\mathcal M_t^{+\infty}(a,b)= \max\{a,b\}$.
A Borel measure $\mu$ on $\R^n$ is said to be \emph{$\kappa-$concave}, where $-\infty\leq \kappa\leq +\infty$, if it verifies the following inequality for all Borel sets $A, B\subset\R^n$ and $t\in [0,1]$:
\begin{equation}\label{eq:concavemeasure}
\mu (tA+(1-t)B)\geq \mathcal{M}_{t}^{\kappa}(\mu(A),\mu(B)).
\end{equation}
When $\kappa = 0$, then $\mu$ is a log-concave measure, and the case $\kappa = -\infty$ corresponds to the largest (by H\"older's inequality) class of measures, called \emph{convex\text} or \emph{hyperbolic\text}  measures.

The characterization of $\kappa$-concave measures is given by the functional versions of the Brunn-Minkowski inequality. The functional form of the a-dimensional inequality~\eqref{eq:BMineq} is the celebrated \emph{Pr\'ekopa-Leindler inequality}~\cite{Lein,Pre1,Pre2}, whereas the dimensional inequality~\eqref{eq:Brunn-Minkowski} is associated to a family of inequalities, known as the \emph{Borell-Brascamp-Lieb inequalities} (BBL in short)  obtained in~\cite{BraLie,Bor2}.  Actually, many of the applications of the Brunn-Minkowski inequality and of  their functional forms, the Pr\'ekopa-Leindler and BBL inequalities, can be obtained from the particular case  where the sets $A$ and $B$ are convex, or when the functions under study are convex. So we shall state these particular cases only. The reader can find more background and applications in~\cite{Gard,Led}. 

The functional form of the a-dimensional inequality~\eqref{eq:BMineq} for convex sets is the celebrated Pr\'ekopa inequality~\cite{Pre2}, which is the following particular case of the Pr\'ekopa-Leindler inequality.

\begin{theorem}{\rm(Pr\'ekopa's inequality)}\label{Prth}
Let $\varphi: \R^{n+1}\to \R\cup\{+\infty\}$ be a convex function. Then the function $\phi$ defined on $\R$  by
$$ e^{-\phi(t)}=\int_{\R^n}e^{-\varphi(t,x)}\, dx, $$
is convex on $\R$.
\end{theorem}

Note that we recover indeed the geometric result~\eqref{eq:BMineq} when $A$ and $B$ are convex sets of $\R^n$ by taking $e^{-\varphi(t,x)}= \mathbf 1_{(1-t)A+tB}(x)$. 

The corresponding dimensional version, relevant for the study of $\kappa$-concave measures with $\kappa\neq 0$,  is the following particular case of the BBL inequality. Accordingly, we shall call it the dimensional Pr\'ekopa or Pr\'ekopa-BBL inequality. It contains two cases.
%
\begin{theorem}{\rm(Pr\'{e}kopa-BBL or dimensional Pr\'{e}kopa inequality)}\label{extensionofPrekopa}

First case: Let $\varphi: \R^{n+1} \to (0,\infty]$ be a positive convex function  and let $\beta>n$. Then the function $\phi$ defined on $\R$ by
$$ \phi(t)= \bigg{(}\int_{\R^n}\varphi(t,x)^{-\beta}\, dx\bigg{)}^{-\frac{1}{\beta-n}},$$
is convex.

Second case:  let $\varphi$ be a positive concave function on $\Omega$,  a \emph{bounded} open convex subset of $\R^{n+1}$, and let  $\beta\geq 0$. Then the function $\psi$ on $\R$ defined by
$$ \psi(t)=\bigg{(}\int_{\Omega(t)}\varphi(t,x)^{\beta}\, dx\bigg{)}^{\frac{1}{\beta+n}}, $$
is concave, where $\Omega(t)=\{x\in \R^n: (t,x)\in \Omega\}$.
\end{theorem}

Of course, in the first statement of Theorem~\ref{extensionofPrekopa}, by modifying $\varphi$ if necessary, we can replace $\R^{n+1}$ by any open convex subset $\Omega$ and hence the integration is taken on the sections $\Omega(t)$. 

Let us mention, for completeness, the geometric consequences of these inequalities in term of Brunn-Minkowski inequalities.  Note that one needs the general BBL inequality if one wants~\eqref{eq:concavemeasure} for all sets; with the particular case recalled above, the reader can check that one gets such inequality for convex sets $A$ and $B$. It follows from the BBL inequality, and from a reverse statement of Borell~(see~\cite{Bor1,Bor2}), that a measure $\mu$ on $\R^n$ absolutely continuous with the Lebesgue measure is $\kappa$-concave~\eqref{eq:concavemeasure}  \emph{if and only if}  $\kappa\le \frac1n$ and $\mu$ is supported on some (open) convex subset $\Omega\subseteq\R^n$ where it has a positive density $p(x)$ which satisfies, for all $t\in (0,1)$,
\begin{equation}\label{eq:densityinequality}
p(tx+(1-t)y)\geq \mathcal{M}_t^{\kappa_n}(p(x),p(y)), \quad \forall\, x, y\in \Omega,
\end{equation} 
where $\kappa_n=\frac{\kappa}{1-n\kappa}\in [-\frac1n,+\infty]$ (equivalently, $\kappa=\frac{\kappa_n}{1+n\kappa_n}\in [-\infty, \frac1n]$). In particular, $\mu$ is log-concave if and only if it has a log-concave density ($\kappa=\kappa_n=0$), which is of course consistent with Pr\'ekopa's inequality. Note that that the Lebesgue measure has the best possible concavity $\kappa=\frac1n$ (which gives the Brunn-Minkowski inequality\eqref{eq:Brunn-Minkowski}) among convex measures, since a constant function satisfies~\eqref{eq:densityinequality} with $\kappa_n=+\infty$.

This description suggests two different behaviors, since,  depending on the sign of $\kappa_n$, $p^{\kappa_n}$ is convex or concave (observe that $\kappa_n$ is nonnegative if and only if $\kappa\in [0, \frac1n]$). It was also noticed by  Bobkov~\cite{Bobkov} that in the case $\kappa \ge 0$, the measures have bounded support. Since these two situations are present all along the paper (and already in the theorems above), let us clearly identify them:
 
{\bf Case 1}: This corresponds to $\kappa \leq 0$. We set   $\beta =-\frac{1}{\kappa_n}= n-\frac{1}{\kappa}\geq n$ and we work with densities $p(x)=\varphi(x)^{-\beta}$  where $\varphi$ is a convex function on $\R^n$ or on a subset $\Omega$.
The typical examples are the (generalized) Cauchy probability measures given by
\begin{equation}\label{eq:Cauchydensity}
d\tau_{\beta}=\frac{1}{Z_{\beta}}(1+|x|^2)^{-\beta}dx,\quad \beta >\frac{n}{2},
\end{equation}
  where $Z_\beta$ is a normalizing constant
$ Z_{\beta}:=\int_{\R^n}(1+|x|^2)^{-\beta}dx= \pi^{\frac{n}{2}}\frac{\Gamma(\beta-\frac{n}{2})}{\Gamma(\beta)}$.

{\bf Case 2}: This corresponds to $0< \kappa\leq \frac{1}{n}$. We set $\beta= \frac{1}{\kappa_n} \in [0,+\infty)$, although later we will also allow $\beta\in (-1, +\infty)$,  and we work with densities $p(x) = \varphi(x)^{\beta}$ where $\varphi$ is a concave function with compact support $\Omega\subset \R^n$. In this case, the typical examples are the probability measures given by 
\begin{equation}\label{eq:Cauchytypedensity}
d\tau_{\sigma,\beta}=Z_{\sigma,\beta}^{-1}(\sigma^2-|x|^2)^{\beta}\ind_{\{|x|\leq \sigma\}}dx,\quad \beta \geq 0, \sigma >0,
\end{equation}
with normalizing constant $Z_{\sigma,\beta}$ given by
$ Z_{\sigma,\beta}=\sigma^{2\beta+n}\pi^{\frac{n}{2}}\frac{\Gamma(\beta+1)}{\Gamma(\beta+\frac{n}{2}+1)}$.

Of major interests for us are the Poincar\'e-Sobolev inequalities that can be deduced from the functional forms of the Brunn-Minkowski inequalities above, as done by Bobkov and Ledoux in~\cite{BobLed, BobLed1, BobLed2} by amplifying a linearization argument due to  Maurey~\cite{Mau}. In~\cite{BobLed1} Bobkov and Ledoux  explained how to derive from the Pr\'ekopa-Leindler inequality the so-called variance Brascamp-Lieb inequality~\cite{BraLie}  which states that for a log-concave probability measure $d\mu=e^{-V}dx$, with $V$ smooth strictly convex on $\R^n$, one has, for every locally Lipschitz function $f$, 
 \begin{equation}\label{eq:BLinequality}
\va_{\mu}(f)\leq \int_{\R^n}\langle (D^2V)^{-1}\nabla f,\nabla f\rangle\, d\mu.
\end{equation}

The dimensional counterpart of this inequality, recently obtained in~\cite{BobLed} as a consequence of the BBL inequality is as follows. Let $\beta > n$ and let $\mu_\beta$ be a probability measure on a convex domain $\Omega$ of the form $d\mu_\beta=\varphi(x)^{-\beta}dx$ where $\varphi$ is a positive convex function on $\Omega$. Then, for any smooth function $f$ on $\Omega$, setting $g=\varphi f$, we have
\begin{equation}\label{eq:localversion}
(\beta+1)\va_{\mu_\beta}(f)\leq \int \frac{\langle(D^2\varphi)^{-1}\nabla g,\nabla g \rangle}{\varphi}\, d\mu_\beta +\frac{n}{\beta-n}\bigg(\int f\, d\mu_\beta\bigg)^2.
\end{equation}
Note that this inequality corresponds to the particular case $r=0$ in~\eqref{eq:convexmainresult} so that the connection between our main theorems and Brunn-Minkowski inequalities is now materializing. 


Interestingly enough, the particular cases the Pr\'ekopa-Leindler and BBL inequalities recalled in Theorems~\ref{Prth} and~\ref{extensionofPrekopa} are sufficient to derive  the two variance inequalities above.
This is somehow at the heart of the local approaches to  Pr\'ekopa's inequality (Theorem~\ref{Prth}), amounting to compute $\phi''(t)$, as explained in~\cite{CorKlar}.
Our original goal was to give such a local approach to the dimensional version (Theorem~\ref{extensionofPrekopa}).
As expected, the local variance inequality associated to Theorem~\ref{extensionofPrekopa} (for the first statement) is the inequality~\eqref{eq:localversion} obtained by Bobkov and Ledoux as a consequence of the BBL inequality. 

So let us first explain in details  the equivalence between the variance inequality~\eqref{eq:localversion} (together with~\eqref{eq:localversion1} below) and the results of Theorem~\ref{extensionofPrekopa}. 
Because there are two cases, corresponding to  $\kappa\leq 0$ or $\kappa >0$, there will be two local variance inequalities. We will treat  the case $\kappa\leq 0$; the same arguments hold for the case $\kappa>0$. By a direct computation, the second derivative of the function $\phi(t)$ of Theorem~\ref{extensionofPrekopa} satisfies
\begin{align}\label{eq:twicedifferentiation}
\frac{\beta-n}{\beta}\frac{\phi''(t)}{\phi(t)}&= \int \frac{\partial_{tt}\varphi(t,x)}{\varphi(t,x)}\, d\mu_t(x)+\frac{n}{\beta-n}\bigg{(}\int \frac{\partial_t\varphi(t,x)}{\varphi(t,x)}\, d\mu_t(x)\bigg{)}^2\notag \\ 
&\ \ \ \ -(\beta+1)\va_{\mu_t}\bigl{(}\frac{\partial_t\varphi(t,\cdot)}{\varphi(t,\cdot)}\bigl{)},
\end{align}
where $\mu_t$ is the probability measure on $\R^n$ given by
$$d\mu_t(x) = \frac{\varphi(t,x)^{-\beta}\, dx}{\int_{\R^n}\varphi(t,\cdot)^{-\beta}}.$$ 
In order to prove~\eqref{eq:localversion}, we can assume for simplicity that $\Omega$ is relatively compact and that $f$ and $\varphi$ are smooth in $\overline{\Omega}$. For $g=\varphi f$, and  $\epsilon>0$ , the natural extension of $\varphi$ with derivative $g$ (to which we add a small uniformly convex factor for convenience) is the function 
$$ \varphi_\epsilon(t,x) :=\varphi(x)+tg(x)+\frac{t^2}{2}\langle(D^2\varphi(x))^{-1}\nabla g(x),\nabla g(x)\rangle +\frac{\epsilon}{2}(|x|^2+t^2). $$
This function is convex on $\Omega\times (-a,a)$ for some $a>0$ small enough depending on $\varphi, g, \epsilon$ and $\Omega$, since $D^2 \varphi(t,x)_{|t=0} \ge \epsilon\, \id$ on $\Omega$, and it satisfies 
$${\varphi_\epsilon} _{|t=0}= \varphi(x)+\epsilon |x|^2/2, \quad \partial_t {\varphi_\epsilon}_{|t=0}= g(x), \quad \partial^2_{tt} {\varphi_\epsilon}_{|t=0}= \langle(D^2\varphi(x))^{-1}\nabla g(x),\nabla g(x)\rangle  + \epsilon .$$
Theorem~\ref{extensionofPrekopa} tell us that the corresponding $\phi=\phi_\epsilon$ is convex. Combining $\phi''(0)\geq 0$ with~\eqref{eq:twicedifferentiation} for $\varphi_\epsilon(t,x)$, and then letting $\epsilon\to 0$, we get, by uniform convergence on $\Omega$,  the inequality~\eqref{eq:localversion}.
Conversely, in order to prove the first statement of Theorem~\ref{extensionofPrekopa}, we can assume by approximation that $\varphi$ is smooth and strictly convex in $x$. Then, if the inequality~\eqref{eq:localversion} holds,  applying it with $g:=\partial_t \varphi_{|t=0}$ and using the fact that
$$\partial^2_{tt}\varphi \geq \langle (D_x^2\varphi)^{-1}\nabla_x\partial_t\varphi,\nabla_x\partial_t\varphi\rangle,$$ 
when $\varphi$ is a convex function of $(t,x)$ (strictly convex in $x$), we  get exactly that $\phi''(t)\geq 0$.  We thus have shown that the inequality~\eqref{eq:localversion} is equivalent to the dimensional Pr\'ekopa's inequality in the case $\kappa\leq 0$. 

Similarly, in the case $\kappa>0$, the local form of the dimensional Pr\'{e}kopa inequality is the following variance inequality: Let $\Omega$ be a bounded open convex subset of $\R^n$ and let $d\mu=\varphi(x)^{\beta}dx$ be a probability measure on $\Omega$, where $\varphi$ is a positive concave function on $\Omega$ and $\beta \geq 0$. Then for any smooth function $f$ on $\Omega$, setting $g=f\varphi$ one has:
\begin{equation}\label{eq:localversion1}
(\beta-1)\va_{\mu}(f)\leq \int_{\Omega}\frac{\langle (-D^2\varphi)^{-1}\nabla g,\nabla g\rangle}{\varphi}\, d\mu +\frac{n}{n+\beta}\bigg(\int_{\Omega}f\, d\mu\bigg)^2.
\end{equation} 

Therefore, it is sufficient to prove the inequalities~\eqref{eq:localversion} and~\eqref{eq:localversion1} to  prove the dimensional Pr\'ekopa's dimenensioanal inequalities of Theorem \ref{extensionofPrekopa}. 
 Our  proof  is based on H\"{o}rmander's $L^2-$method which is known to be useful in the context of variance inequalities and Pr\'ekopa's Theorem. Indeed,  the H\"{o}rmander's $L^2-$method was first used for the local proof of Pr\'ekopa type inequalities by Cordero-Erausquin in~\cite{Cor},  in connection with  Berndtsson's complex generalization~\cite{Berndtsson1998} of  Pr\'ekopa's theorem. As we saw, and as explained in~\cite{CorKlar}, the variance Brascamp-Lieb inequality~\eqref{eq:BLinequality} can clearly be identified as the local form of Pr\'ekopa's inequality, and this variance inequality can of course  be proved by H\"ormander's $L^2$ method (it is exactly a real version of H\"ordmander's $L^2$-estimate \cite{H}).
 
The dimensional versions under study require, however, a few new arguments, as we shall see. 
 And it allows  for the more general statements given in  Theorems~\ref{maintheorem1} and~\ref{maintheorem2}. The inequalities~\eqref{eq:localversion} and~\eqref{eq:localversion1} are particular cases of these theorems when we pick $r = 0$. The case $r=1$ is also of particular interest, as it amounts to weighted Brascamp-Lieb inequalities of the form
$$ \va_{\mu_\beta}(f)\leq \frac{1}{\beta -1}\int_\Omega\langle (D^2\varphi)^{-1}\nabla f,\nabla f\rangle \varphi\, d\mu_\beta.
$$
One can recover from this the classical Brascamp-Lieb inequality~\eqref{eq:BLinequality} for log-concave probability measures.

%

The rest of this paper is organized as follows. In the main, next section, we give the $L^2-$proof  for the variance inequalities~\eqref{eq:convexmainresult} and~\eqref{eq:concavemainresult}.
In Section \S 3, we explain how these inequalities imply reverse H\"older inequalities for negative-$p$-norms $\|\varphi\|_{L^{-p}(dx)}$ of a convex function (Case 1), and $p$-norms in the case of a concave function (Case 2), as obtained by Borell in~\cite{Bor3};  we also present a sharp bound for $\va_{\mu}(V)$ when $d\mu = e^{-V(x)} \, dx$ is a log-concave measure on $\R^n$. Section \S4 discusses weighted Brascamp-Lieb inequalities with application to log-concave measures.
In  Section \S 5, we derive sharp weighted Poincar\'{e}  inequalities for generalized Cauchy type measures. In the last section \S 6, after some general comments, we  explain how the results of the paper automatically extend to a Riemannian manifold $M$ provided on introduces the correct Barky-Emery type tensor associated to the Hessian of $\varphi$ and the Ricci curvature of $M$.


\section{The $L^2-$proof of Theorems \ref{maintheorem1} and \ref{maintheorem2}}
In this section, we  give the proof of the Theorems \ref{maintheorem1} and~\ref{maintheorem2}.
It is inspired by  H\"{o}rmander's $L^2$-duality method. Note that this gives, in particular,  a new proof of the variance inequality~\eqref{eq:localversion} due to Bobkov and Ledoux, and of the inequality~\eqref{eq:localversion1}. 
We will detail the proof of Theorem \ref{maintheorem1}. The proof of Theorem \ref{maintheorem2} is completely similar.

\begin{proof}[Proof of Theorem \ref{maintheorem1}.]
Although the general argument is easy to follow, some of the formulas below are a bit long. The reader is encouraged to set $r=0$ in the present proof, which corresponds to the case of inequality~\eqref{eq:localversion}. Formulas are nicer, and all the interesting ingredients are already at work in this particular case. Also, some of the formulas are complicated by the fact that we have a boundary term. Making formally $\Omega= \R^n$ also simplifies things significantly.

 In order to prove the inequality~\eqref{eq:convexmainresult}, we can assume , by standard approximation arguments, that the domain $\Omega$ is bounded with $C^\infty-$smooth boundary, and $\Omega$ is given by some $C^\infty$-smooth, convex function $\rho: \R^n\to \R$,
$$ \Omega=\{x: \rho(x) < 0\}, \quad\text{ and }\quad \nabla\rho\not=0 \text{ on }\partial\Omega.$$  
We shall denote
$$ \nu(x) = \frac{\nabla\rho(x)}{|\nabla\rho(x)|}$$
the outer normal vector to $\partial\Omega$ at the point $x\in \partial\Omega$. We can also assume that $f$ and $\varphi$ are $C^\infty$ smooth in $\overline \Omega$.

Given $\beta, r$ satisfying the condition of Theorem \ref{maintheorem1}, it is easy to check that $\beta > 2r$. Let us introduce the operator $L$ on $L^2(\mu_\beta)$ given by 
$$ Lu = \varphi^r\Delta u - (\beta - r)\varphi^{r-1}\langle\nabla\varphi, \nabla u\rangle. $$
It is well defined for functions in $C^2(\overline{\Omega})$ (we don't need to discuss the precise domain of $L$ here).
%
Integration by parts gives us that, for all $u\in C^2(\overline{\Omega})$, and $v\in C^1(\overline{\Omega})$,
\begin{equation}\label{eq:integrationbypart}
 \int_\Omega v(x)Lu(x)\, d\mu_\beta = -\int_\Omega \langle\nabla u(x),\nabla v(x)\rangle\, \varphi(x)^ rd\mu_\beta + \int_{\partial\Omega}\frac{\partial u(x)}{\partial \nu(x)}v(x)\varphi(x)^{-\beta + r}dx.
\end{equation}
Next, we need to commute $\nabla$ and $L$. It is readily  checked that for any $1\leq i\leq n$ and $u\in C^\infty(\overline{\Omega})$,
$$\partial_iLu=L\partial_iu+r\varphi^{r-1}\partial_i\varphi\Delta u -(\beta-r)\varphi^{r-1}\sum_{j=1}^n\partial_{ij}\varphi\partial_ju - (\beta-r)(r-1)\varphi^{r-2}\langle\nabla\varphi,\nabla u\rangle\partial_i\varphi .$$
Hence, if $u$ is smooth on $\overline{\Omega}$ and $\frac{\partial u}{\partial \nu} = 0 $ on $\partial\Omega$, one has
$$\int_\Omega (Lu)^2 \, d\mu_\beta = - \int_\Omega \langle\nabla Lu,\nabla u\rangle\, \varphi^r d\mu_\beta,$$
and therefore
\begin{align*}
\int_\Omega (Lu)^2 \, d\mu_\beta& =-\sum_{i=1}^n\int_\Omega L(\partial_iu)\, \varphi^r\partial_iu\, d\mu_\beta -r\int_\Omega \frac{\langle\nabla \varphi,\nabla u\rangle}{\varphi}\Delta u\, \varphi^{2r}d\mu_\beta\\
&\quad + (\beta-r)\int_\Omega\frac{\langle D^2\varphi\nabla u,\nabla u\rangle}{\varphi}\varphi^{2r} d\mu_\beta +(\beta-r)(r-1)\int_\Omega\frac{\langle\nabla\varphi,\nabla u\rangle^2}{\varphi^2}\varphi^{2r} d\mu_\beta\\
&= r\int_\Omega \frac{\langle D^2u\nabla u,\nabla\varphi\rangle}{\varphi}\varphi^{2r} d\mu_\beta -r \int \frac{\langle\nabla \varphi,\nabla u\rangle}{\varphi}\Delta u\, \varphi^{2r} d\mu_\beta\\
&\quad + (\beta-r)\int_\Omega\frac{\langle D^2\varphi\nabla u,\nabla u\rangle}{\varphi}\varphi^{2r} d\mu_\beta +(\beta-r)(r-1)\int_\Omega\frac{\langle\nabla\varphi,\nabla u\rangle^2}{\varphi^2}\varphi^{2r} d\mu_\beta\\
&\quad + \int_\Omega ||D^2u||_{HS}^2\, \varphi^{2r} d\mu_\beta-\int_{\partial\Omega}\frac{\langle D^2u\nabla u,\nabla\rho\rangle}{|\nabla\rho|}\varphi^{-\beta + 2r}dx.
\end{align*}
Moreover,
\begin{align*}
\int_\Omega\frac{\langle D^2u\nabla u,\nabla\varphi\rangle}{\varphi}\varphi^{2r} d\mu_\beta &= \int_\Omega \frac{\langle\nabla \varphi,\nabla u\rangle}{\varphi}\Delta u\, \varphi^{2r}d\mu_\beta + \frac{1}{\beta- 2r} \int_\Omega ||D^2u||_{HS}^2\, \varphi^{2r} d\mu_\beta\\ 
&\ \ - \frac{1}{\beta-2r}\int_\Omega (\Delta u)^2\varphi^{2r}d\mu_\beta -\frac{1}{\beta - 2r}\int_{\partial\Omega}\frac{\langle D^2u\nabla u,\nabla\rho\rangle}{|\nabla\rho|}\varphi^{-\beta + 2r}dx.
\end{align*}
Since $\langle\nabla u(x),\nabla\rho(x)\rangle = 0$ on $\partial\Omega$, we have
$$ \langle D^2u(x)\nabla u(x),\nabla\rho(x)\rangle = - \langle D^2\rho(x)\nabla u(x), \nabla u(x)\rangle,\quad\forall \ x\in \partial\Omega. $$
Therefore,
\begin{align}\label{eq:square}
\int_\Omega (Lu)^2\, d\mu& = \frac{\beta - r}{\beta - 2r} \int_\Omega ||D^2u||_{HS}^2\, \varphi^{2r}d\mu_\beta - \frac{r}{\beta-2r}\int_\Omega (\Delta u)^2\varphi^{2r}d\mu_\beta \notag\\ 
& \ \ \ \  + (\beta-r)\int_\Omega\frac{\langle D^2\varphi\nabla u,\nabla u\rangle}{\varphi}\varphi^{2r} d\mu_\beta  + (\beta-r)(r-1)\int_\Omega\frac{\langle\nabla\varphi,\nabla u\rangle^2}{\varphi^2}\varphi^{2r} d\mu_\beta \notag \\
&\ \ \ \ + \frac{\beta - r}{\beta -2r}\int_{\partial\Omega}\frac{\langle D^2\rho\nabla u,\nabla u\rangle}{|\nabla\rho|}\varphi^{-\beta + 2r}dx.
\end{align}

So fix now a smooth function $f$ on $\overline \Omega$ and denote $\mu_\beta(f) := \int_\Omega fd\mu_\beta$. 
We will use a classical fact concerning solution of the Laplace equation in $L^2(\mu_{\beta-r})$ where 
$$d\mu_{\beta-r}= \varphi ^r \, d\mu_\beta = \varphi^{-(\beta-r)} \, dx = e^{-\log(\varphi^{\beta-r}(x))}\, dx$$ is a measure (not normalized) with smooth positive density on $\overline \Omega$ (this will be the reason for which it is convenient to assume $\Omega$ bounded and smooth). Namely, it follows from classical theory of elliptic equations (see \cite[Theorem $2.5$]{KM} and the references therein), that given a smooth function $F$ on $\overline \Omega$ with $\int F \, d\mu_{\beta-r} = 0$, there exists a function $u\in C^{\infty}(\overline{\Omega})$ with $\frac{\partial u(x)}{\partial \nu(x)} = 0$ on $\partial\Omega$ such that
$$Nu:= \Delta u - \nabla[\log(\varphi^{\beta-r})]\cdot \nabla u = F .$$
We apply this result to $F:=(f-\mu_\beta (f)) \times \varphi^{-r} $, and we get a function $u\in C^{\infty}(\overline{\Omega})$ with $\frac{\partial u(x)}{\partial \nu(x)} = 0$ on $\partial\Omega$ 
such that 
$$Lu = \varphi^r Nu= f- \mu_\beta(f).$$
We will use $u$ to dualize the  inequality. 

Set $\alpha = (\beta - 1)/(\beta -2r + 1)$. We have
\begin{align*}
 \va_{\mu_\beta}(f) 
&= (1 + \alpha)\int_\Omega (f - \mu_\beta(f))Lu\, d\mu_\beta - \alpha \int_\Omega (Lu)^2\, d\mu_\beta.
\end{align*}
Since $g = f\varphi^{1-r}$, one has
$$ \nabla f =\varphi^{r-1}\nabla g + (r - 1)Lu \frac{\nabla\varphi}{\varphi} + (r - 1)\mu_\beta(f) \frac{\nabla\varphi}{\varphi}.$$
Hence we have
\begin{align}\label{squareintegration}
\va_{\mu_\beta}(f)
& = - (1 + \alpha)\int_\Omega \langle \nabla f,\nabla u\rangle\varphi^r d\mu_\beta - \alpha\int_\Omega (Lu)^2\, d\mu_\beta \notag\\ 
& = -(1 + \alpha)\int_\Omega \frac{\langle\nabla u,\nabla g\rangle}{\varphi}\varphi^{2r} d\mu_\beta - (1 + \alpha)(r - 1)\int_\Omega Lu\frac{\langle\nabla\varphi,\nabla u\rangle}{\varphi}\varphi^r d\mu_\beta \notag \\
&\qquad - \alpha \frac{\beta - r}{\beta - 2r} \int_\Omega ||D^2u||_{HS}^2\, \varphi^{2r}d\mu_\beta - (1 + \alpha)(r - 1)\mu_\beta(f) \int_\Omega \frac{\langle\nabla\varphi,\nabla u\rangle}{\varphi}\varphi ^rd\mu_\beta\notag\\
&\quad - \alpha(\beta-r)(r-1)\int_\Omega\frac{\langle\nabla\varphi,\nabla u\rangle^2}{\varphi^2}\varphi^{2r} d\mu_\beta - \alpha(\beta-r)\int_\Omega\frac{\langle D^2\varphi\nabla u,\nabla u\rangle}{\varphi}\varphi^{2r} d\mu_\beta\notag\\
&\quad  + \alpha\frac{r}{\beta-2r}\int_\Omega (\Delta u)^2\varphi^{2r}d\mu_\beta -\alpha \frac{\beta - r}{\beta -2r}\int_{\partial\Omega}\frac{\langle D^2\rho\nabla u,\nabla u\rangle}{|\nabla\rho|}\varphi^{-\beta + 2r}dx \notag\\
&= -(1 + \alpha)\int_\Omega \frac{\langle\nabla u,\nabla g\rangle}{\varphi}\varphi^{2r} d\mu_\beta - \alpha(\beta-r)\int_\Omega\frac{\langle D^2\varphi\nabla u,\nabla u\rangle}{\varphi}\varphi^{2r} d\mu_\beta \notag\\
&\qquad - \alpha \frac{\beta - r}{\beta - 2r} \int_\Omega ||D^2u||_{HS}^2\varphi^{2r} d\mu_\beta - (1 + \alpha)(r - 1)\mu_\beta(f) \int_\Omega \frac{\langle\nabla\varphi,\nabla u\rangle}{\varphi}\varphi^r d\mu_\beta\notag\\
&\quad  -(1 + \alpha)(r - 1)\int_\Omega \Delta u\frac{\langle\nabla\varphi,\nabla u\rangle}{\varphi}\varphi^{2r} d\mu_\beta + (\beta-r)(r-1)\int_\Omega\frac{\langle\nabla\varphi,\nabla u\rangle^2}{\varphi^2}\varphi^{2r} d\mu_\beta\notag \\
&\quad  +\alpha\frac{r}{\beta-2r}\int_\Omega (\Delta u)^2\varphi^{2r}d\mu_\beta -\alpha\frac{\beta - r}{\beta -2r}\int_{\partial\Omega}\frac{\langle D^2\rho\nabla u,\nabla u\rangle}{|\nabla\rho|}\varphi^{-\beta + 2r}dx .
\end{align}
We now calculate the term $\int_\Omega\frac{\langle\nabla\varphi,\nabla u\rangle^2}{\varphi^2}\varphi^{2r} d\mu_\beta$. Denote $\gamma = \beta -2r$, it follows from the definition of $L$ that
\begin{align*}
\int_\Omega\frac{\langle\nabla\varphi,\nabla u\rangle^2}{\varphi^2}\varphi^{2r} d\mu_\beta& = \frac{1}{(\beta - r)^2}\bigg( \int_\Omega (Lu)^2 d\mu_\beta - \int (\Delta u)^2\varphi^{2r} d\mu_\beta\bigg) \\
&\ + \frac{2}{\beta - r}\int_\Omega \Delta u\, \frac{\langle\nabla\varphi,\nabla u\rangle}{\varphi}\varphi^{2r} d\mu_\beta\\
& = \frac{1}{\gamma(\gamma + r)}\int_\Omega ||D^2u||_{HS}^2\, \varphi^{2r}d\mu_\beta - \frac{1}{\gamma(\gamma + r)}\int_\Omega (\Delta u)^2\varphi^{2r} d\mu_\beta \\ 
&\ + \frac{1}{\gamma + r} \int_\Omega\frac{\langle D^2\varphi\nabla u,\nabla u\rangle}{\varphi}\varphi^{2r} d\mu_\beta + \frac{2}{\gamma + r}\int_\Omega \Delta u\frac{\langle\nabla\varphi,\nabla u\rangle}{\varphi}\varphi^{2r} d\mu_\beta \\
&\ +\frac{r - 1}{\gamma + r}\int_\Omega \frac{\langle\nabla\varphi,\nabla u\rangle^2}{\varphi^2}\varphi^{2r} d\mu_\beta + \frac{1}{\gamma(\gamma + r)}\int_{\partial\Omega}\frac{\langle D^2\rho\nabla u,\nabla u\rangle}{|\nabla\rho|}\varphi^{-\beta + 2r}dx. 
\end{align*}
Equivalently, we have
\begin{align}\label{abcd}
\int_\Omega\frac{\langle\nabla\varphi,\nabla u\rangle^2}{\varphi^2}\varphi^{2r} d\mu_\beta& =\frac{1}{\gamma(\gamma + 1)}\int_\Omega ||D^2u||_{HS}^2\, \varphi^{2r}d\mu_\beta - \frac{1}{\gamma(\gamma + 1)}\int_\Omega (\Delta u)^2\varphi^{2r} d\mu_\beta \notag\\ 
&\ + \frac{1}{\gamma + 1} \int_\Omega\frac{\langle D^2\varphi\nabla u,\nabla u\rangle}{\varphi}\varphi^{2r} d\mu_\beta + \frac{2}{\gamma + 1}\int_\Omega \Delta u\frac{\langle\nabla\varphi,\nabla u\rangle}{\varphi}\varphi^{2r} d\mu_\beta\notag \\ 
&\ + \frac{1}{\gamma(\gamma + 1)}\int_{\partial\Omega}\frac{\langle D^2\rho\nabla u,\nabla u\rangle}{|\nabla\rho|}\varphi^{-\beta + 2r}dx.
\end{align}
It follows from~\eqref{eq:integrationbypart} that $\int_\Omega Lu\, d\mu_{\beta} = 0$, or equivalently
\begin{equation}\label{eq:vho}
\int_\Omega \frac{\langle\nabla\varphi,\nabla u\rangle}{\varphi}\varphi^r d\mu_\beta = \frac{1}{\beta - r}\int_\Omega \Delta u\, \varphi^r d\mu_\beta.
\end{equation}
Combining~\eqref{squareintegration},~\eqref{abcd} and~\eqref{eq:vho} with the value $\alpha = (\beta - 1)/(\beta - 2r + 1)$, one has
\begin{align}\label{squareintegration3}
(\beta - 2r + 1)\va_{\mu_\beta}(f)& =  -2(\beta - r)\int_\Omega \frac{\langle\nabla u,\nabla g\rangle}{\varphi}\varphi^{2r} d\mu_\beta - (\beta - r)^2\int_\Omega\frac{\langle D^2\varphi\nabla u,\nabla u\rangle}{\varphi}\varphi^{2r} d\mu_\beta \notag \\ 
& \quad - \frac{(\beta - r)^2}{\beta - 2r}\int_\Omega ||D^2u||_{HS}^2\, \varphi^{2r}d\mu_\beta + \frac{\beta -2r + r^2}{\beta - 2r}\int_\Omega (\Delta u)^2\varphi^{2r}d\mu_\beta\notag \\
&\quad - 2(r - 1)\mu_\beta(f) \int_\Omega \varphi^r\Delta u \, d\mu_\beta\notag \\
&\quad - \frac{(\beta - r)^2}{\beta - 2r}\int_{\partial\Omega}\frac{\langle D^2\rho\nabla u,\nabla u\rangle}{|\nabla\rho|}\varphi^{-\beta + 2r}dx.
\end{align}
If $\beta, r$ satisfy the condition of Theorem \ref{maintheorem1}, then $\beta > 2r$, hence
$$ \frac{(\beta - r)^2}{\beta - 2r}\int_{\partial\Omega}\frac{\langle D^2\rho\nabla u,\nabla u\rangle}{|\nabla\rho|}\varphi^{-\beta + 2r}dx \geq 0.$$ 
By using the pointwise estimates, $2\langle v,w\rangle -\langle Hv,v\rangle \leq \langle H^{-1}w,w\rangle$ and $(\tr(Q))^2 \leq n||Q||_{HS}^2$ for two vector $v, w\in \R^n$, $H$ a positive $n\times n$ matrix and $Q$ a symmetric $n\times n$ matrix, one gets
$$ -2(\beta - r)\langle\nabla u,\nabla g\rangle - (\beta - r)^2\langle D^2\varphi\nabla u,\nabla u\rangle \leq \langle (D^2\varphi)^{-1}\nabla g,\nabla g\rangle, $$
and, since $\beta > 2r$,
$$  - \frac{(\beta - r)^2}{\beta - 2r}||D^2u||_{HS}^2 + \frac{\beta -2r + r^2}{\beta - 2r}(\Delta u)^2\leq -A(n,\beta,r) (\Delta u)^2. $$
Moreover, one has $A(n,\beta,r) > 0$ when $\beta > r + (n + \sqrt{n^2+ 4(r^2 - r)n})/2 $ and so  
\begin{align*}
(\beta -2r +1)\va_{\mu_\beta}(f) &\leq \int_\Omega \frac{\langle (D^2\varphi)^{-1}\nabla g,\nabla g\rangle}{\varphi}\varphi^{2r}d\mu_\beta - A(n,\beta,r)\int_\Omega (\varphi^r\Delta u)^2d\mu_\beta\\
&\qquad -2( r- 1)\mu_\beta(f) \int_\omega \varphi^r\Delta u\, d\mu_\beta\\ 
&\leq \int_\Omega \frac{\langle (D^2\varphi)^{-1}\nabla g,\nabla g\rangle}{\varphi}\varphi^{2r}d\mu_\beta + \frac{(1 - r)^2}{A(n,\beta, r)}\bigg(\int_\Omega f\, d\mu_\beta\bigg)^2.
\end{align*}
This finishes the proof of Theorem \ref{maintheorem1}.

\end{proof}


\section{Reverse H\"older inequalities and convexity}

We give here some very elementary applications of Theorem~\ref{maintheorem1} and Theorem \ref{maintheorem2} obtained by taking $g=1$ in the inequalities. We shall detail the applications of Theorem \ref{maintheorem1}, and state some results without proof for Theorem \ref{maintheorem2}. 

So let us take $\beta > r + (n + \sqrt{n^2 + 4(r^2 - r)n})/ 2$ and a convex function $\varphi > 0$ on an open convex set $\Omega\subseteq \R^n$ (we drop the normalization $\int \varphi(x)^{-\beta}dx=1$). Setting $d\mu_\beta(x) = \frac{\varphi^{-\beta} dx}{\int \varphi^{-\beta}}\, dx$, we can rewrite the inequality~\eqref{eq:convexmainresult} as follows: when $\varphi$ is smooth and $f$ is a locally Lipschitz function $f\in L^2(\mu_\beta)$, we have
\begin{equation}\label{eq:main3}
R(f)\le \frac{1}{\beta -2r +1}\int_\Omega \frac{\langle (D^2\varphi)^{-1}\nabla g,\nabla g\rangle}{\varphi}\, d\mu_\beta,
\end{equation}
with $g=\varphi^{1 - r} f$,  and
\begin{equation}\label{eq:defR}
R(f):=\int_\Omega f^2\, d\mu_\beta -\big(1 + \frac{(1 - r)^2}{(\beta -2r + 1)A(n,\beta,r)}\big)\bigg(\int_\Omega f\, d\mu_\beta\bigg)^2.
\end{equation}
Observe that if we take the $g$ identically one in \eqref{eq:main3}, we get that $R(\varphi^{r-1})\le 0$.
We deduce:

\begin{proposition}\label{Phiconcave}
Let $\Omega\subseteq \R^n$ be a convex open set and $\varphi$ be a positive convex function on $\Omega$. If $\beta > r + (n +\sqrt{n^2 + 4(r^2 - r)n})/2$, then
\begin{equation}\label{eq:concave}
\int_\Omega \varphi^{-\beta} dx\int_\Omega \varphi^{-\beta-2(1 - r)} dx\leq \bigl(1 + \frac{(1 - r)^2}{(\beta -2r + 1)A(n,\beta,r)}\bigl)\bigg(\int_\Omega\varphi^{-\beta-1+ r} dx\bigg)^2.
\end{equation}
In particular (case $r=0)$, setting
\begin{equation}\label{eq:Psifunction}
\Psi(\beta) := \ln\bigg(\prod_{i=1}^n(\beta - i)\int_\Omega\varphi^{-\beta}dx\bigg)
\end{equation}
we have
\begin{equation}\label{eq:psi}
\Psi(\beta)+\Psi(\beta+2)\leq 2\Psi(\beta +1), \quad\forall\, \beta>n.
\end{equation}
\end{proposition}

\begin{proof}
By approximation, we can assume that the convex function $\varphi$ is smooth and strictly convex on $\Omega$. Then~\eqref{eq:concave} is exactly the property $R(\varphi^{r-1})\le 0$ that we deduced from plugging $g=1$ in~\eqref{eq:main3}. When $r=0$ the inequality rewrites as 
\begin{equation}\label{eq:psi2}
\int \varphi^{-\beta} dx\int \varphi^{-\beta-2} dx\leq \frac{\beta(\beta-n+1)}{(\beta+1)(\beta-n)}\bigg(\int\varphi^{-\beta-1} dx\bigg)^2,
\end{equation}
which is equivalent to~\eqref{eq:psi}.
\end{proof}

It is interesting to note that there is equality in~\eqref{eq:psi}-\eqref{eq:psi2} when $\varphi$ comes from a $1$-homogeneous function, for instance in the following way. When $\Omega = \R^n$, if we take $\varphi(x)=(1+J_C(x))$ whith $J_C$ being the gauge of a convex body $C\subset \R^n$, then equality holds in~\eqref{eq:psi}-\eqref{eq:psi2}. Indeed, one then has for $\beta>n$,  $\int \varphi^{-\beta}= c_{n,\beta} |C|$ where $c_{n,\beta}$ depends on $n$ and $\beta$ only, and therefore can be computed using $\varphi(x) = 1+ |x|$, for which one gets the relation~\eqref{eq:psi} by elementary calculus.  Note that this also shows that the constant $A(n,\beta,0)=\beta/(\beta-n)$ is optimal in~\eqref{eq:localversion}. In the case $\Omega$ is bounded, the argument works also when $C$ is chosen to be a multiple of $\overline\Omega$.

Inequality~\eqref{eq:psi} suggests that $\Psi$ might be concave and is reminiscent of the Berwald type inequalities obtained by Borell (in the Case 2, see below). Let us point out that the concavity of $\Psi$ is  stated (without proof) by Bobkov and Madiman in~\cite{BobMad1}. A weaker, though useful,  concavity can be deduced from~\eqref{eq:concave} as follows. 
Let us define the function $\psi$ on $(n,\infty)$ by
$$\psi(\beta)=\ln \bigg(\int_\Omega \varphi^{-\beta}dx\bigg).$$
Inequality~\eqref{eq:concave} is equivalent to
\begin{equation}\label{eq:logconcave2}
\psi(\beta)+\psi(\beta+2(1 - r)) - 2\psi(\beta + 1 -r)\leq \ln\bigg(1 + \frac{(1 - r)^2}{(\beta -2r + 1)A(n,\beta,r)}\bigg), 
\end{equation}
for all $\beta> r + (n +\sqrt{n^2 + 4(r^2 - r)n})/2$.  Since $r + (n +\sqrt{n^2 + 4(r^2 - r)n})/2\to n+ 1$ when $r\to 1$, we have for any $\beta > n+ 1$,  that inequality~\eqref{eq:logconcave2} holds for all $r$ which is close enough to $1$. Dividing the two sides of~\eqref{eq:logconcave2} by $(1-r)^2$ and then letting $r\to 1$, we get 
\begin{equation}\label{eq:sederivative}
\psi''(\beta)\leq \frac{n(\beta -2)}{(\beta -1)^2(\beta - n -1)},\quad\forall\, \beta > n+1.
\end{equation}
Therefore we have an upper bound for second derivative of the convex function $\psi$ on $(n+1,\infty)$. Moreover, it is readily checked that~\eqref{eq:sederivative} is equivalent to the concavity of the  function 
\begin{equation}\label{eq:Phifunction}
\Phi(\beta) = \ln\bigg((\beta -1)\int_\Omega \varphi^{-\beta}dx\bigg) -\frac{n-1}{n}\ln\bigg( \frac{(\beta -1)^{\beta -1}}{(\beta - n - 1)^{\beta - n -1}}\bigg).
\end{equation}
It is possible to improve inequality~\eqref{eq:sederivative} in particular cases, when $\varphi$ is smooth and strictly convex. Indeed,  by using  inequality~\eqref{eq:mainoptimization} with functions of the form  $f = \varphi^{\alpha(r-1)}$, $\alpha \not=1$, we  can get the following inequality
$$ \psi''(\beta)\leq \frac{W(\varphi, \beta)}{1 + W(\varphi, \beta)}\cdot \frac{n(\beta -2)}{(\beta -1)^2(\beta -n -1)},$$
with
$ W(\varphi, \beta ) := \frac{(\beta -1)(\beta -n - 1)}{n(\beta -2)\int_\Omega\varphi^{-\beta}}\int_\Omega\frac{\langle (D^2\varphi)^{-1}\nabla\varphi,\nabla\varphi\rangle}{\varphi}\varphi^{-\beta}\, dx$. 
This improves the bound~\eqref{eq:sederivative} when  $W(\varphi , \beta) < \infty$.

Inequality~\eqref{eq:sederivative} is weaker, in general, than the concavity of $\Psi$, except in dimension $n=1$ where $\Phi=\Psi$. Nonetheless,  it allows for a sharp variance estimate improving a result of Bobkov and Madiman. 

%
%

\begin{corollary}
Let $d\mu=e^{-V(x)}dx$ be a log-concave probability measure on $\R^n$. Then
\begin{equation}\label{eq:varianceofinformation}
\va_{\mu}(V)=\int V(x)^2e^{-V(x)}dx-\bigg(\int V(x) e^{-V(x)}dx\bigg)^2\leq n.
\end{equation}
\end{corollary}
\begin{proof}
Note that with the notation $\psi(\beta)=\log\int \varphi^{-\beta}$ we have
 $ \psi''(\beta) = \va_{\mu_{\beta}}(\ln \varphi), $where $\mu_{\beta}$ is  probability measure defined by
$$d\mu_{\beta}=\frac{\varphi(x)^{-\beta}dx}{\int e^{-\beta}dx}.$$
In our case, let $V(x)$ be a convex function on $\R^n$ such that $\int e^{-V}dx=1$. Fix $\beta_0 > n+1$ and apply the inequality~\eqref{eq:sederivative} to the convex function $\varphi = e^{V/\beta_0}$ at $\beta=\beta_0$. We get 
$$ \int V(x)^2e^{-V(x)}dx-\bigg(\int V(x)e^{-V(x)}dx\bigg)^2\leq  \beta_0^2\frac{n(\beta_0 -2)}{(\beta_0 -1)^2(\beta_0 -n -1)}.$$
Letting $\beta_0$ tend to infinity, one gets the following variance inequality for $V$
$$ \int V(x)^2e^{-V(x)}dx- \bigg(\int V(x)e^{-V(x)}dx\bigg)^2\leq n.
$$ as claimed.
\end{proof}
Inequality~\eqref{eq:varianceofinformation} was obtained in~\cite{BobMad} by Bobkov and Madiman with an universal constant $C\neq 1$ multiplying the $n$. Our version is sharp, as one can verify that there is equality in~\eqref{eq:varianceofinformation} for the exponential distribution $e^{-\sum |x_i|}/2^n$. Furthermore, note that when $d\mu=e^{-V(x)}dx$ is an isotropic log-concave probability measure on $\R^n$, that is $\int xe^{-V(x)}dx=0$ and $\int x\otimes x e^{-V(x)}dx= I_n$ then one has, by using the Cauchy-Schwartz inequality,
\begin{align*}
n=\int \langle x,\nabla V(x)\rangle e^{-V(x)}dx& \leq \bigg(\int |x|^2e^{-V(x)}dx\bigg)^{\frac{1}{2}}\bigg(\int |\nabla V(x)|^2e^{-V(x)}dx\bigg)^{\frac{1}{2}}\\
&=\sqrt{n}\bigg(\int |\nabla V(x)|^2e^{-V(x)}dx\bigg)^{\frac{1}{2}},
\end{align*}
and so the inequality~\eqref{eq:varianceofinformation} rewrites in this case as
$$\va_{\mu}(V) \leq \int |\nabla V(x)|^2e^{-V(x)}dx.$$

We now mention similar consequences in the Case 2, that can be derived from Theorem~\ref{maintheorem2}. Recall that in this situtation $\varphi$ is a positive concave function on a \emph{bounded}, open , convex subset $\Omega\subset\R^n$ and introduce the probability measure supported on $\Omega$
$$d\nu_\beta(x) = \frac{\varphi(x)^{\beta}\ind_{\Omega}(x)}{\int_\Omega \varphi^\beta}\, dx$$
Note that
$\int_\Omega \varphi(x)^{\alpha}dx < \infty, $
for all $\alpha > -1$, and so $\beta$ is a priori allowed to range in $(-1, +\infty)$. Let us denote
\begin{equation}\label{eq:defRbar}
\overline{R}(f)=\int_{\Omega}f^2\, d\nu_\beta -\big(1 + \frac{(1 - r)^2}{(\beta +2r -1)B(n,\beta,r)}\big)\bigg(\int_{\Omega} f \, d\nu_\beta\bigg)^2, 
\end{equation}
As above, when $\varphi$ is smooth, we rewrite~\eqref{eq:concavemainresult} as $(\beta +2r -1)\overline{R}(f) \le \int_\Omega \frac{\langle (-D^2\varphi)^{-1}\nabla g,\nabla g\rangle}{\varphi}\,\varphi^{2r} d\nu_\beta$ for any smooth $f$ and $g=\varphi^{1-r} f$. In particular, if we take $f=\varphi^{r-1}, g=1$, we get $(\beta +2r -1)\overline{R}(\varphi^{r-1}) \le 0$. This of course extends to general $\varphi$ and we get:
\begin{proposition} 
Let $\Omega$ be a convex body of $\R^n$ and $\varphi$ be a positive convex function on $\Omega$.
If $\beta > \max\{-r + (-n +\sqrt{n^2+4(r^2-r)n})/2, 1-2r\}$, then 
\begin{equation}\label{eq:concavecase}
\int_{\Omega}\varphi^{\beta}dx\int_{\Omega}\varphi^{\beta+ 2r -2}dx\leq  \big(1 + \frac{(1 - r)^2}{(\beta +2r -1)B(n,\beta,r)}\big)\bigg{(}\int_\Omega \varphi^{\beta+ r-1}dx\bigg{)}^2.
\end{equation} 
In particular (case $r=0)$, setting
$$
\overline\Psi(\beta) := \ln\bigg(\prod_{i=1}^n(\beta + i)\int_\Omega\varphi(x)^{\beta}dx\bigg)
$$
we have
\begin{equation}\label{eq:psibar}
\overline\Psi(\beta)+\overline\Psi(\beta+2)\leq 2\overline\Psi(\beta +1), \quad\forall\, \beta > -1.
\end{equation}

\end{proposition}
Inequality~\eqref{eq:psibar} is a special case of a result of Borell~\cite{Bor3} who proved that $\overline\Psi$ is concave on $(0,+\infty)$. As before, one can show that inequality~\eqref{eq:concavecase} implies, by letting $r\to 1$, the weaker result that the function $\overline{\Phi}$ defined on $(-1,\infty)$  by
$$ \overline{\Phi}(\beta) =\ln\bigg((\beta +1)\int_\Omega\varphi^\beta dx\bigg) -\frac{n-1}{n}\ln\bigg(\frac{(\beta+1)^{\beta+1}}{(\beta +n+1)^{\beta+n+1}}\bigg)$$
is concave. Note however, that in dimension $n=1$ this reproduces and extends the result of Borell, since it gives the concavity of $\Psi$ in the larger range $(-1, +\infty)$.
 Let us mention that the concavity of $\overline{\Phi}$ in the form $\overline\Phi'' \le 0$ can be used to reproduce the inequality~\eqref{eq:varianceofinformation} as well.


\section{Some weighted Brascamp-Lieb inequalities and applications}

The following Brascamp-Lieb-type inequality can be derived from the Theorem \ref{maintheorem1},
\begin{theorem}\label{BLtypeineq}
Let $\varphi$ be a $C^2$, positive, convex function defined on an (open) convex subset $\Omega\subseteq \R^n$. For any $\beta > n$, we denote $\mu_\beta$ the probability measure on $\Omega$ given by 
$ d\mu_\beta(x) = \frac{\varphi(x)^{-\beta}}{\int_\Omega\varphi^{-\beta}}dx.$ 
Then, when $\beta\geq n+1$, we have that for any locally Lipschitz function $f\in L^2(\mu_\beta)$, 
\begin{equation}\label{eq:BLtypeineq}
\va_{\mu_\beta}(f)\leq \frac{1}{\beta -1}\int_\Omega\langle (D^2\varphi)^{-1}\nabla f,\nabla f\rangle\, \varphi\, d\mu_\beta.
\end{equation}
\end{theorem}
\begin{proof}
For the case $\beta > n+1$, our result is proved by using the Theorem \ref{maintheorem1} with $r =1$ and function $\widetilde{\varphi} = c\varphi$ with $c^\beta =\int_\Omega\varphi^{-\beta}dx$.

The case $\beta = n+1$ is proved by letting $\beta$ decrease to $n+1$.
\end{proof}
Furthermore, one can derive from~\eqref{eq:BLtypeineq} applied to $\beta +1$, after proper normalization, the following reverse-weighted  inequality: for any  locally Lipschitz function $f$ on $\Omega$,
 \begin{equation}\label{eq:reverseBLtypineq}
\inf\limits_{c\in \R}\int \frac{|f(x)-c|^2}{\varphi(x)}\, d\mu_{\beta}(x)
\leq \frac{1}{\beta}\int \langle (D^2\varphi)^{-1}\nabla f,\nabla f\rangle\, d\mu_{\beta},\quad\, \forall\, \beta\geq n.
\end{equation}

Similarly, by applying the Theorem \ref{maintheorem2} to $r =1$, one gets
\begin{theorem}\label{BLtypeineq1}
Let $\varphi$ be a positive concave function on a compact, convex set $\Omega\subset\R^n$. For $\beta > -1$, denote $\nu_\beta $ the probability measure on $\Omega$ defined by
$ d\nu_\beta(x) = \frac{\varphi(x)^{\beta}}{\int_\Omega\varphi^\beta}dx. $
Then for any locally Lipschitz function $f\in L^2(\nu_\beta)$, we have 
\begin{equation}\label{eq:BLtypeineq1}
\va_{\nu_\beta}(f)\leq \frac{1}{\beta + 1}\int_\Omega\langle(-D^2\varphi)^{-1}\nabla f,\nabla f\rangle\, \varphi\, d\nu_{\beta}.
\end{equation}
\end{theorem}

Moreover, for any bounded, smooth function $f$ on $\Omega$ and $\beta > 0$, the following reversed-weighted form of~\eqref{eq:BLtypeineq1} holds
\begin{equation}\label{eq:reverseBLtypineq1}
\inf_{c\in \R}\int_{\Omega}\frac{|f(x)-c|^2}{\varphi}\, d\nu_\beta\leq \frac{1}{\beta}\int_\Omega\langle( -D^2\varphi)^{-1}\nabla f,\nabla f\rangle\, d\nu_{\beta},
\end{equation}

Inequalities~\eqref{eq:BLtypeineq} and~\eqref{eq:BLtypeineq1} allow to simplify some arguments given by Bobkov and Ledoux~\cite{BobLed} on how to recove the~Brascamp-Lieb inequality~\eqref{eq:BLinequality}. 


Let $d\mu = e^{-V(x)}dx$ be a log-concave probability measure on $\R^n$. We can assume that $V$ is bounded from below.  Hence $1+ V/\beta$ is a positive, convex function for $\beta $ large enough.  Denoting 
$$ c_{\beta} = \bigg(\int (1 + \frac{V(x)}{\beta})^{-\beta}dx\bigg)^{\frac{1}{\beta}},$$
and applying  inequality~\eqref{eq:BLtypeineq} to the convex function $\varphi(x) = c_{\beta}(1 + V(x)/\beta)$ with $\beta$ large enough, one obtains
$$ \int f^2\bigl(1 + \frac{V}{\beta}\bigl)^{-\beta}\frac{dx}{c_{\beta}^{\beta}} -\bigg(\int f\bigl(1 + \frac{V}{\beta}\bigl)^{-\beta}\frac{dx}{c_{\beta}^{\beta}}\bigg)^2\leq \frac{\beta}{\beta - 1}\int \langle (D^2V)^{-1}\nabla f,\nabla f\rangle (1+\frac{V}{\beta})^{-\beta+1} \frac{dx}{c_{\beta}^{\beta}},$$
for any bounded, smooth function $f$ on $\R^n$. Letting $\beta$ tend to infinity, one obtains~\eqref{eq:BLinequality} (since $\lim\limits_{\beta\to\infty}c_{\beta}^{\beta} = 1$).

We can use~\eqref{eq:BLtypeineq1} as well. For this,  consider the concave function $1 - V/\beta$ on the domain $\Omega_\beta =\{x\ | \ V(x)\leq \beta\}$. If we let  $\beta\to\infty$ then we get again the inequality~\eqref{eq:BLinequality}.

There is another way of obtaining~\eqref{eq:BLinequality} from~\eqref{eq:BLtypeineq}  which is of independent interest and goes through an slight improvement of an inequality of Bobkov and Ledoux.
\begin{proposition}
Let $\mu$ be a log-concave probability measure on $\R^n$ with density $e^{-V}$ where $V$ is a smooth convex function, and let $\beta\ge n+1$.  For $x\in \R^n$, write $W_x$ for the nonnegative self-adjoint operator
$$ W_x := D^2 V(x) + \frac1\beta \nabla V(x) \otimes \nabla V(x),$$
where $a\otimes a$ denotes the linear map $x\to \langle x, a\rangle a$. Then, for any locally Lipschitz function $f$ we have 
\begin{equation}\label{eq:application}
\va_{\mu}(f)\leq \frac{\beta}{\beta-1}  \int \langle W^{-1} \nabla f , \nabla f\rangle \, d\mu.
\end{equation}
\end{proposition}
\begin{proof}
 For our $\beta > n$, we define a convex function $\varphi = e^{V/\beta}$. Then $D^2\varphi = \bigl(\frac{1}{\beta}D^2V + \frac{1}{\beta^2}\nabla V\otimes \nabla V\bigl)e^{V/\beta}$.
 Applying~\eqref{eq:BLtypeineq} for $\beta\geq n + 1$, one gets the result.
 \end{proof}
This result was first proved by Bobkov and Ledoux~\cite{BobLed} with the (worse) constant $C_\beta := (\sqrt{\beta +1}+1)^2 /\beta$ in place of $\beta/(\beta-1)$ (although on the larger range $\beta>n$). 
Since for the one-rank perturbation of a positive matrix we have 
\begin{equation}\label{eq:reversematrix}
\bigl(A + a\otimes a)^{-1} = A^{-1}-\frac{A^{-1}a\otimes A^{-1}a}{1+\langle A^{-1}a,a\rangle},
\end{equation}
we can rewrite the inequality as
$$
\va_{\mu}(f)\leq \frac{\beta}{\beta-1}\bigg{(}\int\langle(D^2 V)^{-1}\nabla f,\nabla f\rangle e^{-V}dx -  \int \frac{\langle(D^2 V)^{-1}\nabla V,\nabla f\rangle^2}{\beta+ \langle(D^2 V)^{-1}\nabla V,\nabla V\rangle} e^{-V}dx\bigg{)},
$$
for any bounded, smooth function $f$ on $\R^n$. And Brascamp-Lieb inequality~\eqref{eq:BLinequality} is deduced from~\eqref{eq:application} by letting $\beta\to\infty$. 

As in~\cite{BobLed}, an application of~\eqref{eq:application} to Gaussian measures on $\R^n$ gives us the weighted Poincar\'e type inequality for the family of $\chi_n-$distributions on $[0,\infty)$ defined by
$$ d\chi_n(r) =  \frac{2^{1-\frac{n}{2}}}{\Gamma(\frac{n}{2})}r^{n-1}e^{-\frac{r^2}{2}}1_{[0,\infty)}.$$
Indeed, for any bounded, smooth function $g$ on $[0,\infty)$, setting $f(x) = g(|x|)$ and $\beta = n + 1$, inequality~\eqref{eq:application} with $V(x) = (|x|^2 + n\ln 2\pi)/2$ yeilds
$$ \va_{\chi_n}(g)\leq \frac{(n+1)^2}{n}\int_0^{\infty}\frac{(g'(r))^2}{n+1+r^2}\, d\chi_n\leq \int_0^{\infty}(g'(r))^2\frac{n+3}{n+r^2}\, d\chi_n. $$

Another interesting application of inequality~\eqref{eq:application} concerns the probability measures on $\R^n$ having the density $d\mu_{r,n}(x) = c_{r,n}\exp\{-(|x_1|^r + \cdots + |x|^r)/ r\}dx$ with $r\in [1,2]$.
In particular, the result below reproduces the Poincar\'{e} inequality (although with non-optimal numerical constants) for such measures.
\begin{proposition}
Let $ r \in [1,2]$ and $f$ be a smooth, $\mu_{r,n}-$square integrable function on $\R^n$. Then the following inequality holds:
\begin{eqnarray*}
\int f(x)^2\, d\mu_{r,n} - \bigg(\int f(x)\, d\mu_{r,n}\bigg)^2 & \le & 4 \int \sum_{i=1}^n 
\frac{|x_i|^{2-r}}{|x_i|^r + 2(r-1)} \bigg( \frac{\partial f}{\partial x_i}(x)\bigg)^2 \, d\mu_{r,n}(x) \\
&\leq& C_r\int |\nabla f(x)|^2\, d\mu_{r,n}, 
\end{eqnarray*}
 with 
$C_r = \frac{4}{r}(2 - r)^{\frac{2-r}{r}} \in (\frac95, 4]$.
\end{proposition}
\begin{proof}
Since $\mu_{r,n} = \mu_{r,1}\otimes\cdots\otimes\mu_{r,1}$ it is enough, by elementary tenzoration, to prove the inequality in dimension $n=1$. 
We consider the case $r > 1$; the case $r = 1$ follows by limit. Applying the inequality~\eqref{eq:application} for $\beta = 2$ and  to the convex function $V(x) = |x|^r/r - \ln c_{r,1}$ on $\R$, one has
$$ \int_{\R}f(x)^2\, d\mu_{r,1} - \bigg(\int_\R f(x)\, d\mu_{r,1}\bigg)^2 \leq 4\int_\R \frac{|x|^{2-r}}{2(r - 1) + |x|^r}\bigl(f'(x)\bigl)^2\, d\mu_{r,1}. $$
Since $1< r\leq 2$ then the function $g(t) = t^{2-r}/(2(r - 1) + t^r)$ is bounded on $[0,\infty)$, with $g(t)\leq (2 - r)^{\frac{2-r}{r}}/r$ for all $t\geq 0$. 
\end{proof}


\section{Weighted Poincar\'{e} inequality for uniformly convex potentials with application to the Cauchy measures}

This section discusses weighted Poincar\'{e} inequality for some special probability measures which are the Cauchy measures $\tau_{\beta}$ defined by~\eqref{eq:Cauchydensity} and the measures $\tau_{\sigma,\beta}$ defined by~\eqref{eq:Cauchytypedensity}. 

We observe first that if $\varphi $ is uniformly convex on $\Omega$, that is $D^2\varphi(x) \geq C I_n$ for all $x\in \Omega$ and for some $C > 0$, where $I_n$ denotes $n\times n$ identity matrix,  then we get from the Theorem~\ref{BLtypeineq} 
the following weighted Poincar\'{e}-type inequality and its reverse-weighted form. 

\begin{theorem}\label{weightedPoincare}
Let $\varphi$ be a positive, strictly convex function on an open convex set  $\Omega\subseteq \R^n$, such that $D^2\varphi\geq CI_n$, for some $C > 0$. Introduce the probability measure on $\Omega$ given by
$ d\mu_\beta = \frac{\varphi(x)^{-\beta}}{\int_\Omega\varphi^{-\beta}}dx $.
 Then, when $\beta \ge n+1$, we have,  for any locally Lipschitz function $f\in L^2(\mu_\beta)$, that
\begin{equation}\label{eq:weightedPoincare}
\va_{\mu_{\beta}}(f)\leq \frac{1}{C(\beta - 1)}\int |\nabla f(x)|^2\,\varphi(x)\, d\mu_{\beta}.
\end{equation}
Moreover, for any $\beta\geq n$ and for any smooth, bounded function $f$ on $\Omega$, it holds
\begin{equation}\label{eq:reverseweighted}
\inf\limits_{c\in \R}\int \frac{|f(x)-c|^2}{\varphi(x)}\, d\mu_{\beta}(x)\leq \frac{1}{C\beta}\int |\nabla f|^2\, d\mu_{\beta}.
\end{equation}
\end{theorem}
It is well-known that the Poincar\'{e} inequality is equivalent to the exponential convergence of the semi-group with the generator associated to the Dirichlet form. Inequality~\eqref{eq:weightedPoincare} possesses a similar property, more precisely:
\begin{proposition}\label{spectral-gap}
Let $\varphi$ be a convex function satisfying the conditions of Theorem \ref{weightedPoincare} and  $\beta\geq n+1$. Denote 
 $P_t = e^{tL_\beta}$ is the semigroup associated to the differential operator $L_\beta := \varphi\Delta -(\beta - 1)\langle\nabla\varphi,\nabla \cdot\, \rangle$ on $L^2(\mu_\beta)$.Then  inequality~\eqref{eq:weightedPoincare} is equivalent to
\begin{equation}\label{eq:spectral-gap}
\va_{\mu_\beta}(P_tf)\leq e^{-2C(\beta -1)t}\va_{\mu_\beta}(f), 
\end{equation}
for any $f\in L^2(\mu_\beta)$.
\end{proposition}
\begin{proof}
We will give a formal proof of this proposition (see \cite{Bakry1992} for more precise justification of the computations involving $P_t$, $L_\beta$ and the domain of $L_\beta$). 

Assume that~\eqref{eq:weightedPoincare} holds. Since $P_t1 = 1$ and $\int P_tf d\mu_\beta = \int f\,d\mu_\beta$, then it is sufficient to prove~\eqref{eq:spectral-gap} for $f\in L^2(\mu_\beta)$ and $\int f d\mu_\beta =0$. We define $F(t) = \int (P_tf)^2d\mu_\beta$. Then the derivative of $F$ satisfies $ F'(t) \leq -2C(\beta -1)F(t)$ by using the inequality~\eqref{eq:weightedPoincare}. This inequality proves~\eqref{eq:spectral-gap}.

Conversely, assume that~\eqref{eq:spectral-gap} holds. Since~\eqref{eq:spectral-gap} becomes an equality at $t=0$, differentiating the two sides of~\eqref{eq:spectral-gap} at $t=0$ gives~\eqref{eq:weightedPoincare}. 
\end{proof}

Let us return to the Cauchy measures $\tau_\beta$ defined by~\eqref{eq:Cauchydensity}; for these measures, we have $\varphi(x) = 1 + |x|^2$, hence $D^2\varphi = 2I_n$. For instance, when $ \beta > r + (n + \sqrt{n^2 + 4(r^2 -r)n})/2$, \eqref{eq:convexmainresult} takes the following form: for any locally Lipschitz $f$ and $g=f\varphi^{1-r}$,
$$(\beta - 2r +1)\va_{\tau_\beta}(f)\leq \frac12 \int_\Omega |\nabla g|^2\, \varphi^{2r-1}\, d\tau_\beta + \frac{(1 - r)^2}{A(n,\beta,r)} \bigg(\int_\Omega f\, d\tau_\beta\bigg)^2. $$
Let us consider the particular case ($r=1$) given by the Theorem~\ref{weightedPoincare} and the Proposition~\ref{spectral-gap}:

\begin{corollary}\label{Cauchymeasure}
Let $\beta\geq n+1$. For any locally Lipschitz $f\in L^2(\tau_\beta)$ we have
\begin{equation}\label{eq:Cauchymeasure}
\va_{\tau_{\beta}}(f)\leq \frac{1}{2(\beta-1)}\int |\nabla f(x)|^2(1+|x|^2)\, d\tau_{\beta}.
\end{equation}
Moreover, if $\beta\geq n$ then
$$\inf\limits_{c\in \R}\int \frac{|f(x)-c|^2}{1+|x|^2}\, d\tau_{\beta}\leq \frac{1}{2\beta} \int | \nabla f(x)|^2\, d\tau_{\beta}.$$
Finally, let us denote $L_\beta = (1+|x|^2)\Delta -2(\beta - 1)\langle x,\nabla \rangle$ and $P_t$ be the semigroup associated to $L_\beta$ on $L^2(\tau_\beta)$, then
$$ \va_{\tau_\beta}(P_tf)\leq e^{-4(\beta -1)t}\va_{\tau_\beta}(f) $$
for any $f\in L^2(\tau_\beta)$ and $\beta\geq n+1$.
\end{corollary}

The weighted Poincar\'{e}-type inequality~\eqref{eq:Cauchymeasure} improves  a result of Bobkov and Ledoux (Theorem {\bf 3.1}, \cite{BobLed}). In that paper the authors obtained a similar result with the constant $C_{\beta}= (\sqrt{1+\frac{2}{\beta-1}}+\sqrt{\frac{2}{\beta+1}})^2$ in the place of $1$ in the right hand side of~\eqref{eq:Cauchymeasure}. A simple calculation with the linear test functions $f(x)=\langle v_0, x\rangle$  shows the constant $1/2(\beta-1)$ in~\eqref{eq:Cauchymeasure} to be sharp. The disadvantage of the Theorem \ref{Cauchymeasure} compared with the result of Bobkov and Ledoux is that the domain of $\beta$ is smaller, that is $\beta \geq n+1$ instead of $\beta\geq n$ as in \cite{BobLed}.

We now consider the case in which $\varphi$ is positive, concave and $\Omega$ is bounded. If $-\varphi$ is strictly convex, it follows from the Proposition \ref{BLtypeineq1} and the same arguments in the Proposition \ref{spectral-gap} that
\begin{theorem}\label{weightedPoincare1}
Let $\varphi$ be a positive, concave function on a bounded convex set $\Omega\subset \R^n$ such that $-D^2\varphi\geq CI_n$ for some $C > 0$. Introduce the probability measure on $\Omega$ given by $ d\nu_\beta = \frac{\varphi^{\beta}(x)\ind_\Omega(x)}{\int_\Omega\varphi^{\beta}}dx$. Then, when $\beta > -1, $ we have  for any locally Lipschitz  $\nu_{\beta}-$square integrable $f$ on $\Omega$, that
\begin{equation}\label{eq:weightedPoincare1}
\va_{\nu_\beta}(f) \leq \frac{1}{C(\beta + 1)}\int_\Omega |\nabla f(x)|^2\varphi(x)\, d\nu_{\beta}.
\end{equation}
And for any bounded, smooth function $f$ on $\Omega$,
\begin{equation}\label{eq:reverseweightedPoincare}
\inf_{c\in \R}\int_{\Omega}\frac{|f(x)-c|^2}{\varphi}\, d\nu_\beta\leq \frac{1}{C\beta}\int_{\Omega}|\nabla f(x)|^2\, d\nu_{\beta},\quad\forall\, \beta > 0.
\end{equation}
Moreover, let us denote $N_\beta = \varphi\Delta + (\beta + 1)\langle\nabla\varphi,\nabla \rangle$ and let $P_t$ be the semigroup associated to $N_\beta$ on $L^2(\nu_\beta)$. Then we have
$$ \va_{\nu_\beta}(P_tf)\leq e^{-2C(\beta + 1)t}\va_{\nu_\beta}(f),$$
for any function $f\in L^2(\nu_\beta)$ and $\beta > -1$.
\end{theorem}
Let us finally consider the measures $\tau_{\sigma,\beta}$ defined by~\eqref{eq:Cauchytypedensity}. For these measures, one has $\varphi(x) = \sigma^2 - |x|^2$. Since $D^2\varphi = -2 I_n$, applying Theorem~\ref{weightedPoincare1}, we get the following results for measures $\tau_{\sigma,\beta}$,
\begin{corollary}\label{Poincareineq}
Given $\beta > -1$ and $\sigma > 0$,  let $\tau_{\sigma,\beta}$ be the probability measure defined by~\eqref{eq:Cauchytypedensity}. For any locally Lipschitz, $\tau_{\sigma,\beta}-$square integrable functions $f$ on $\Omega = \{x: |x|<\sigma\}$,  we have
\begin{equation}\label{eq:wpi}
\va_{\tau_{\sigma,\beta}}(f)\leq \frac{1}{2(\beta+1)}\int_{\Omega}|\nabla f(x)|^2(\sigma^2-|x|^2)\, d\tau_{\sigma,\beta}.
\end{equation}
Moreover, if $\beta\geq 0$, for any smooth function $f$ on $\Omega$, we get
$$\inf_{c\in \R}\int_{\Omega}\frac{|f(x)-c|^2}{\sigma^2-|x|^2}\, d\tau_{\sigma,\beta}\leq \frac{1}{2\beta}\int |\nabla f(x)|^2\, d\tau_{\sigma,\beta}.$$
Finally, consider $N_{\sigma,\beta} = (\sigma^2 -|x|^2)\Delta -2(\beta + 1)\langle x,\nabla\rangle$ and let $P_t=e^{tN_{\sigma, \beta}}$ be the semigroup associated to $N_{\sigma,\beta}$ on $L^2(\mu_{\sigma,\beta})$; then 
$$ \va_{\tau_{\sigma,\beta}}(P_tf)\leq e^{-4(\beta + 1)t}\va_{\tau_{\sigma,\beta}}(f), $$
for any function $f\in L^2(\tau_{\sigma,\beta})$ and $\beta > -1$.
\end{corollary}
Let us remark that inequality~\eqref{eq:wpi} is sharp, and that equality holds for the linear functions $f(x)=\langle v_0, x\rangle$, for any  $v_0\in \R^n$.


\section{Further remarks}

We conclude with some straightforward extensions of  Theorems~\ref{maintheorem1} and~\ref{maintheorem2}.

One should note that the inequalities in these theorems are not invariant under translation. Consider first the Case 1, with $\varphi$ convex smooth on some open convex set $\Omega\subseteq \R^n$ and $ \beta > r + (n + \sqrt{n^2 + 4(r^2 -r)n})/2$. Recall the definition of $R(f)$ in~\eqref{eq:defR}.
Next introduce
$$ S(f):=\int_\Omega f\varphi^{r-1}\, d\mu_\beta - \big(1 + \frac{(1 - r)^2}{(\beta -2r + 1)A(n,\beta,r)}\big)\int_\Omega f \, d\mu_\beta\int_\Omega\varphi^{r-1}\, d\mu_\beta,$$
so that
$$ R(f+c\varphi^{r -1})=R(f)+2cS(f)+c^2R(\varphi^{r -1}),\quad\forall \, c\in\R. $$
Since the right hand side of~\eqref{eq:main3} does not change if we replace $f$ by $f+c\varphi^{r-1}$, we have
\begin{equation}\label{eq:main4}
R(f)+2cS(f)+c^2R(\varphi^{r-1})\leq \frac{1}{\beta -2r +1}\int \frac{\langle (D^2\varphi)^{-1}\nabla g,\nabla g\rangle}{\varphi}\, \varphi^{2r} d\mu_\beta ,\quad \forall\, c\in\R
\end{equation}
with  $g=\varphi^{1 - r} f$. 
Optimizing the left hand side of~\eqref{eq:main4} over $c\in\R$, we obtain a stronger version of the inequality~\eqref{eq:convexmainresult},
\begin{equation}\label{eq:mainoptimization}
R(f)-\frac{S(f)^2}{R(\varphi^{r -1})}\leq \frac{1}{\beta -2r +1}\int \frac{\langle (D^2\varphi)^{-1}\nabla g,\nabla g\rangle}{\varphi}\, \varphi^{2r} d\mu_\beta ,
\end{equation}
with $g=\varphi^{1-r} f$.

Similarly, in the Case 2, let  $\varphi$ be a concave function defined on a bounded open convex set $\Omega\subset \R^n$ and $ \beta > -r + (-n + \sqrt{n^2 + 4(r^2 -r)n})/2$. Recall the definition of $\overline R(f)$ in ~\eqref{eq:defRbar} and introduce
$$ \overline{S}(f):=\int_{\Omega}f\varphi^{r-1}\, d\nu_\beta-\bigl(1 + \frac{(1 - r)^2}{(\beta +2r -1)B(n,\beta,r)}\bigl)\int_{\Omega}f \, d\nu_\beta\int_{\Omega}\varphi^{r-1}\, d\nu_\beta.$$
The optimized version of~\eqref{eq:concavemainresult} is as follows (since in this case $(\beta +2r -1)\overline{R}(\varphi^{r-1})\leq 0$):
\begin{equation}\label{eq:optconcave}
(\beta+2r -1)\bigg(\overline{R}(f)-\frac{\overline{S}(f)^2}{\overline{R}(\varphi^{r-1})}\bigg)\leq \int_\Omega \frac{\langle (-D^2\varphi)^{-1}\nabla g,\nabla g\rangle}{\varphi}\, \varphi^{2r}d\nu_\beta,
\end{equation} 
with $g=\varphi^{r-1} f$. 

A priori, these optimized forms leave room for different kind of normalization: $\int f d\mu_\beta = 0$, or $\int f \varphi^{r-1} d\mu_\beta = 0$. However, we were not able to obtain new information from it.

A nicer observation, maybe,  is that the results in this paper automatically extend to Riemannian manifolds. This is one of the advantages of the $L^2$ approach we exploited here. 
%

Let $M$ be a complete $n$-dimensional Riemannian manifold, equiped with its Riemannian element of volume $d\vol$. We assume, in addition, that $M$ has the following approximate property: there exists an increasing sequence of the compact subsets $\{M_k\}_k$ of $M$ such that $M = \cup_{k=1}^\infty M_k$ and $M_k$ is given by
$$M_k =\{x\, :\, \rho_k(x) < 0\},\quad\rho_k\,\text{ is smooth, }\quad D^2\rho_k(x) \ge 0,\quad\forall\, x\in M.$$
The only difference in the computations we did in Section \S 2 is an extra Ricci curvature term  coming from the commutation of  the (covariant) derivative and the Laplacian. More precisely, one can check, using the Bochner-Lichnerovicz formula (see for instance Proposition 4.15 in~\cite{GHL}),
 that formula~\eqref{eq:square} holds with an extra term on the right equal to
$\frac{\beta -r}{\beta -2r}\int_M\Ric (\nabla u,\nabla u)\varphi^{2r} d\mu_\beta $
where $\Ric_x (\cdot , \cdot)$ stands for the Ricci curvature at point $x$ (that is also identified to a symmetric operator on the tangent space $T_x M$). And then formula~\eqref{squareintegration3} holds with the term $\frac{(\beta -r)^2}{\beta -2r}\int_M\Ric(\nabla u,\nabla u)\varphi^{2r}d\mu_\beta$ added on its right hand side. Given a smooth function $\varphi$ on $M$, introduce the symmetric operator, which can be seen as a modified Bakry-Emery tensor, defined on $T_x M$ by
$$H_x\varphi := D^2\varphi(x)  + \frac{\varphi(x)}{\beta -2r}\Ric_x ,$$
where $D^2(x)\varphi$ denotes the Riemannian Hessian of $\varphi$ at $x$. 
Then the result of this paper extend to $M$ provided one properly replaces the convexity of $\varphi$ (or of $-\varphi$ for the Case 2). Here is an example of result, corresponding to Theorem~\ref{maintheorem1}
\begin{theorem}
Let $M$ be a complete $n$-dimensional manifold having the approximate property above. Let us give constants $\beta, r\in \R$ and $A(n,\beta,r) $ as in Theorem~\ref{maintheorem1}. Assume that we are given a probability measure $d\mu_\beta(x) = \varphi(x)^{-\beta}d\vol(x)$ where $\varphi$ a smooth function on $M$ such that $H_x \varphi > 0$ at every $x\in M$. Then for any locally Lipschitz $\mu_\beta-$square integrable function $f$ on $M$, setting $g = f\varphi^{1-r}$, we have
$$
(\beta - 2r +1)\va_{\mu_\beta}(f)\leq \int\frac{\langle (H\varphi)^{-1}\nabla g,\nabla g\rangle}{\varphi}\, \varphi^{2r}\, d\mu_\beta + \frac{(1 - r)^2}{A(n,\beta,r)} \bigg(\int f\, d\mu_\beta\bigg)^2. 
$$
\end{theorem}

We leave to the reader the corresponding applications and particular cases given in Sections \S 3 and \S 4. Note that for the Case 2 (where $M$ is bounded) with the measure $d\nu_\beta = \varphi(x)^\beta \, d\vol(x)$, the operator to consider is
$$\widetilde H_x \varphi := - D^2\varphi(x) + \frac{\varphi(x)}{\beta + 2r}\Ric_x \, ,$$
and the concavity of $\varphi$ is replaced by the requirement that  $\widetilde H_x $ is positive at every $x\in M$.


\subsection*{Acknowledgment}
I would like to thank my advisor Dario Cordero-Erausquin for his encouragements, his careful review of this manuscript and his many useful discussions. 



\begin{thebibliography}{99}
\bibitem{Bakry1992}
D. Bakry, \emph{L'hypercontractivit\'e et son utilisation en th\'eorie des semigroupes\text}, Lectures on probability theory (Saint-Flour, 1992), Lecture Notes in Math., 1581, Springer, Berlin. (1994) 1-114.

\bibitem{Berndtsson1998}
B. Berndtsson, 
\emph{Prekopa's theorem and Kiselman's minimum principle for plurisubharmonic functions},
Math. Ann. {\bf 312} (1998),  785--792. 

\bibitem{Bobkov}
S. G. Bobkov, \emph{Large deviations and isoperimetry over convex probability measures with heavy tails\text}, Electron. J. Probab. {\bf 12} (2007) 1072-2100 (eletronic).

\bibitem{BobLed}
S. G. Bobkov and M. Ledoux, \emph{Weighted Poincar\'e-type inequalities for Cauchy and other convex measures\text}, Ann. Probab. {\bf 37} (2009) 403-427.

\bibitem{BobLed1}
S. G. Bobkov and M. Ledoux, \emph{From Brunn-Minkowski to Brascamp-Lieb and to logarithmic Sobolev inequalities\text}, Geo. Funct. Anal. {\bf 10} (2000) 1028-1052.

\bibitem{BobLed2}
S. G. Bobkov and M. Ledoux, \emph{From Brunn-Minkowski to Sharp Sobolev inequalities}, Ann. Mat. Pura. Appl (4) {\bf 187} (2008) 369-384.

\bibitem{BobMad}
S. G. Bobkov and M. Madiman, \emph{Concentration of the information in data with log-concave distributions\text}, Ann. Probab. {\bf 39} (2011) 1528-1543.

\bibitem{BobMad1}
S. G. Bobkov and M. Madiman, \emph{The entropy per coordinate of a random vector is highly constrained under convexity conditions\text}, IEEE Trans. Inform. Theory, {\bf 57}(8) (2011) 4940-4954.

\bibitem{Bor1}
C. Borell, \emph{Convex measures on locally convex spaces\text}, Ark. Mat. {\bf 12} (1974) 239-253.

\bibitem{Bor2}
C. Borell, \emph{Convex set functions in $d-$space\text}, Period. Math. Hungar. {\bf 6} (1975) 111-136.

\bibitem{Bor3}
C. Borell, \emph{Complements of Lyapunov's Inequality\text}, Math. Ann. {\bf 205} (1973) 323-331.

\bibitem{BraLie}
H. J. Brascamp and E. H. Lieb, \emph{On extensions of the Brunn-Minkowski and Pr\'ekopa-Leindler theorems, including inequalities for log concave functions, and with an application to the diffusion equation\text}, J. Funct. Anal. {\bf 22} (1976) 366-389.

\bibitem{Cor}
D. Cordero-Erausquin, \emph{On Berndtsson's generalization of Pr\'ekopa's theorem\text}, Math. Z. {\bf 249} (2005) 401-410.

\bibitem{CorFraMau}
D. Cordero-Erausquin, M. Fradelizi and B. Maurey, \emph{The (B) conjecture for the Gaussian measure of dilates of symmetric convex sets and related problems\text}, J. Funct. Anal. {\bf 214} (2004) 410-427.

\bibitem{CorKlar}
D. Cordero-Erausquin and B. Klartag, \emph{Interpolations, convexity and geometric inequalities\text}, Geometric Aspects of Functional Analysis, Springer Lecture Notes in Math. {\bf 2050} (2012), 151--168.


\bibitem{GHL}
S. Gallot, D. Hulin, and J. Lafontaine,  Riemannian geometry. Third edition. Universitext. Springer-Verlag, Berlin, 2004. 

\bibitem{Gard}
R. J. Gardner, \emph{The Brunn-Minkowski inequality\text}, Bull. Amer. Math. Soc. (N.S.) {\bf 39} (2002) 355-405.

\bibitem{H}
L. H\"ormander, \emph{$L^2$ estimate and existence theorems for the $\bar{\partial}$ operator\text}, Acta Math. {\bf 113} (1965) 89-152.

\bibitem{KM}
A. V. Kolesnikov, and E. Milman, \emph{Poincar\'e and Brunn-Minkowski inequalities on weighted Riemannian manifolds with boundary\text}, arxiv.org/abs/1310.2526v2.

\bibitem{Led}
M. Ledoux, \emph{The concentration of measure phenomenon}, American Mathematical Society, Providence, RI, 2001.
 
\bibitem{Lein}
L. Leindler, \emph{On a certain converse of H\"{o}lder's inequality. II\text}, Acta. Sci. Math. (Szeged). {\bf 33} (1972) 217-223.

\bibitem{Mau}
B. Maurey, \emph{Some deviation inequalities\text}, Geo. Funct. Anal. {\bf 1} (1991) 188-197.

\bibitem{Pre1}
A. Pr\'ekopa. \emph{Logarithmic concave measures with application to stochastic programming\text}, Acta. Sci. Math. (Szeged). {\bf 32} (1971) 301-316.

\bibitem{Pre2}
A. Pr\'ekopa. \emph{Logarithmic concave measures and functions\text}, Acta. Sci. Math. (Szeged). {\bf 34} (1973) 335-343.

\end{thebibliography}
\end{document}